\documentclass{amsart}

\theoremstyle{plain}
\newtheorem{Prop}{Proposition}[section]
\newtheorem{Thm}[Prop]{Theorem}
\newtheorem{Cor}[Prop]{Corollary}

\newtheorem{Lem}[Prop]{Lemma}

\theoremstyle{definition}
\newtheorem{Def}[Prop]{Definition}

\theoremstyle{remark}
\newtheorem{Rem}[Prop]{Remark}
\newtheorem{Ex}[Prop]{Example}
\newtheorem{Problem}[Prop]{\bf Problem}

\def\dim{\mathop{\roman{dim}}}
\def\int{\mathop{\roman{int}}}

\def\1{^{-1}}

\def\dim{\text{dim}}
\def\asdim{\text{asdim}}

\def\tf{{\tilde f}}
\def\tg{{\tilde g}}
\def\T2{{\mathbf T_2}}
\def\NN{{\mathbb N}}
\def\QQ{{\mathbb Q}}
\def\UU{{\mathcal U}}
\def\LL{{\mathcal L}}

\def\RR{{\mathbb R}}
\def\ZZ{{\mathbb Z}}
\def\Zmed{\ZZ(\frac{1}{2})}
\def\Ztwo{\ZZ(2^{\infty})}

\errorcontextlines=0 \numberwithin{equation}{section}

\begin{document}
\title[A coarse classification of countable abelian groups]{A coarse classification of countable abelian groups}
\author{J.~Higes}
\address{Departamento de Geometr\'{\i}a y Topolog\'{\i}a,
Facultad de CC.Matem\'aticas. Universidad Complutense de Madrid.
Madrid, 28040 Spain}
\email{josemhiges@yahoo.es}

\date{3 Marzo, 2008}

\thanks{ The author is supported by Grant AP2004-2494 from the Ministerio de Educaci\' on y Ciencia, Spain and project MEC, MTM2006-0825. 
He thanks Kolya Brodskyi, Manuel Alonso Moron 
and Jesus Pascual Moreno Damas for helpful discussions. He also thanks Jose Manuel Rodriguez Sanjurjo for his support. Special thanks to Jerzy Dydak 
for all his help and nice suggestions.}




\maketitle

\begin{abstract}
We classify up to coarse equivalence all countable abelian groups of finite torsion free rank. The $\QQ$-cohomological 
dimension and the torsion free rank and  are the two invariants that give us such classification. We also prove that any countable abelian group of finite 
torsion free rank is coarsely equivalent to $\ZZ^n \oplus H$ where $H$ is a direct sum (possibly infinite) of cyclic groups. A partial generalization 
to countable abelian groups of the Gromov rigidity theorem for abelian groups is shown.     
\end{abstract}

\tableofcontents
\section{Introduction and preliminaries}

Given a finitely generated group $G$ with $L = \{g_1, g_2,..., g_n\}$ a symmetric system of generators of $G$ the norm $\|g\|_G$ of an element $g\in G$ is 
the minimum $k \in \NN$ such that $g =  g_{i_1}\cdot...\cdot g_{i_k}$ with $g_{i_m}\in L$. With this norm one can easily define a metric $d_G$ in $G$ by 
$d_G(g, h) = \|g^{-1}\cdot h\|_G$. Such metric is called a word metric. Gromov introduced in \cite{Gro asym invar} 
some invariants to study the geometry of finitely generated groups with word metrics. One important problem is to classify finitely generated groups 
quasi-isometrically. Notice that as finitely generated groups are quasi-geodesic then a coarse classification of finitely generated groups will imply a 
quasi-isometric classification. Some examples of this kind of results has been obtained for finitely presented groups of asymptotic dimension $1$ 
(\cite{Germitis}, \cite{Fuji}) and for abelian groups(\cite{Gromov2}). But the main problem remains open even for the class of meta-abelian groups
(see \cite{Harpe} for further details).\par
In last years there has been a lot of 
research (see \cite{Brod-Dydak-Higes-Mitra}, \cite{Dran-Smith}, \cite{Higes}, \cite{Sauer}, \cite{Shalom} and \cite{Smith} for example) 
trying to apply the ideas of Gromov to countable groups not necessarily finitely generated. For such groups it is clear that the notion of word metric need to be changed. 
To construct metrics in countable groups that satisfy similar properties as word metrics we use the concept of proper norm. 
A map $\|\cdot\|_G: G \to \RR_+$ is called to be a {\it proper norm} if  it satisfies
the following conditions:
\begin{enumerate}
\item $\|g\|_{G} = 0$ if and only if $g$ is the neutral element of $G$.
\item $\|g\|_G = \|g^{-1}\|_G$ for every $g \in G$.
\item $\|g\cdot h\|_G \le \|g\|_G +\|h\|_G$ for every $g, h \in G$.
\item For every $K>0$ the number of elements of $G$ such that $\|g\|_G\le K$ is finite. (proper condition)
\end{enumerate}
So if we build a proper norm $\|\cdot\|_G$ in $G$ then the map $d_G(g, h) = \|g^{-1}\cdot h\|_G$ defines in a group 
what is called a {\it proper left invariant metric}. 
Recall that a metric $d_G$ defined in a group $G$ is {\it left invariant} if $d_G(g\cdot h_1, g\cdot h_2) = d_G(h_1, h_2)$ for every $g, h_1, h_2 \in G$.
\par 
It is important to find 
(non finitely generated) countable versions of results about finitely generated groups. Notice that 
proper left invariant metrics appear also naturally when we take the restriction $d_G|_H$ of a word metric $d_G$ of a finitely generated group $G$
to some subgroup $H\subset G$. Hence the study of the coarse geometry of countable groups with proper left invariant metrics 
can give us information about the geometry of finitely generated groups with word metrics. One example of this fact is the paper of Dranishnikov and Smith 
(\cite{Dran-Smith}) about the asymptotic dimension of countable groups.\par
Another coarse invariants that has been studied in the context of countable groups are the asymptotic Assouad-Nagata dimension 
in \cite{Brod-Dydak-Higes-Mitra} and \cite{Higes} and the cohomological dimension in \cite{Shalom} and \cite{Sauer}. In this two last papers Shalom and Sauer shown that 
the cohomological dimension is a coarse invariant for certain types of countable groups. Given a ring $R$ the {\it cohomological dimension} 
of a group $G$ over $R$ ($cd_{R}(G)$) is the supremum of all the numbers $n$ such that there exists an $RG$-module $V$ with $H^n(G, V) \ne 0$. 
The following results will be important for us:
\begin{Thm} (Shalom-Sauer \cite{Shalom}, \cite{Sauer})\label{Shalom-Sauer} Let $R$ be a commutative ring and suppose there is an uniform embedding(coarse embedding) 
of $G$ into $H$ with $G$ and $H$ 
discrete countable groups. Then the following two statements hold:
\begin{enumerate}
\item If $cd_R(G)$ is finite then $cd_R(G) \le cd_R(H)$.
\item If $G$ is amenable and $\QQ \subset R$ then $cd_R(G) \le cd_R(H)$.
\end{enumerate}
Futhermore (1) and (2) hold true if the cohomological dimension is replaced by the homological dimension.
\end{Thm}
In view of this result a coarse classification of countable abelian groups could help us to determine their cohomological and homological dimensions 
with respect some rings. 
\par
Recall that a map $f: (X, d_X) \to (Y, d_Y)$ between metric spaces is said to be a {\it coarse map} if for every $\delta >0$ there exists and $\epsilon >0$ such that if $d_X(x, y) \le \delta$ 
then $d_Y(f(x), f(y)) \le \epsilon$. If moreover there exists a coarse map $g: (Y, d_Y) \to (X, d_X)$ ant two positive constants $K_1$ and  $K_2$ such that for every 
$x \in X$ and $y \in Y$ we gave $d_X(g(f(x)), x) \le K_1$ and $d_Y(f(g(y)), y)\le K_2)$ then it is said that $f$ is a {\it coarse equivalence} and $g$ is a 
{\it coarse inverse} of $f$.
If there exists a coarse equivalence between two metric spaces it is said that both spaces are {\it coarsely equivalent}. If a map $f: (X, d_X) \to (Y, d_Y)$ defines a 
coarse equivalence between $X$ and some subspace of $Y$ it is said that $f$ is a {\it coarse embedding}. For a study of coarse map in a more general way check \cite{Roe}.   
\par
The coarse structure of finitely generated abelian groups is well known. Firstable is not hard to see that a finitely generated abelian group $G$ is coarsely equivalent to $\ZZ^n$ with 
$n$ its torsion free rank. A second remarkable (but not so easy) result is the following quasi-isometric rigidity of $\ZZ^n$ (see \cite{Gromov2} and \cite{Shalom}): 
\begin{Thm} (Gromov)\label{QuasiIsometryRigidity} If a group $G$ is quasi-isometric to $\ZZ^n$ then it has a finite index subgroup isomorphic to $\ZZ^n$. 
\end{Thm}
In other words a finitely generated group is coarsely equivalent to an abelian group if and only if it is virtually abelian. Recall that a group $G$ is said to satisfy 
a property $\mathcal P$ virtually if there exists a subgroup $H \subset G$ of finite index that satisfies $\mathcal P$. \par
The main targets of this paper is to 
study countable version of this two results about abelian groups. In particular we will show that any countable(non finitely generated) 
abelian group $G$ of finite torsion free rank is coarsely equivalent to $\ZZ^n \oplus M$ with $M$ a locally finite group of the form $\bigoplus_{i = 1}^{\infty} \ZZ_{p_i}$(Theorem \ref{SuperTheorem}  and Theorem 
\ref{MainTheorem}). Also we will prove that 
if a countable group $G$ is coarsely equivalent to an abelian group of torsion free rank finite then it is locally virtually abelian(Theorem \ref{Rigidity}). A countable group 
is said to satisfy a property $\mathcal P$ {\it locally} if all its finitely generated subgroups satisfy $\mathcal P$. It is unknown 
if the converse of this fact is also true(Problem \ref{RigidityProblem}). \par
To prove the main results we study the coarse splitting (see definition \ref{CoarselySplit}) of exact sequences. In particular we 
are interested in determining which properties has to satisfy an exact sequence $1 \to H \to G \to G/H\to 1$ to guarantee that $G$ is coarsely equivalent to 
$H\oplus G/H$. Notice that in general it is not always true even for finitely presented groups, otherwise nilpotent groups would be coarsely abelian what is clearly false. \par
In Section \ref{Growth} we will study the growth function in countable groups. In Section \ref{Torsion} we will discuss the abelian torsion groups.
Sections \ref{Odd} and \ref{Even} are devoted to the study of the odd part and the even part 
of a countable abelian group and the coarse splitting theorems mentioned early. Finally in section \ref{MainSection} we state the main results of this paper.

\section{Large scale connectedness and growth of countable groups}\label{Growth}
In this section we will show that large scale connectedness is useful to determine if a group is finitely generated or not. As a plus, this will 
leads us to show that the growth of countable groups deppends uniquely of the growth of its finitely generated subgroups. \par
We will define the growth of a countable group using the idea of $s$-scale chain. 
Let $s$ be a positive real number. An {\it $s$-scale chain} of length $m$ (or $s$-path)
 between two points $x$ and $y$ of a metric space $(X, d_X)$ is defined as
a finite sequence points
$\{x= x_0, x_1, ..., x_m = y\}$ such that
$d_X(x_i, x_{i+1}) < s$ for every $i = 0, ..., m-1$.
A subset $S$ of a metric space $(X, d_X)$ is said to be {\it $s$-scale connected} if there exists an $s$-scale chain
contained in $S$ for every two elements of $S$. A metric space $(X, d_X)$ is said to be {\it large scale connected} if it is 
$s$-scale connected for some $s > 0$. It is not difficult to see that large scale connectedness is invariant under coarse equivalences. \par
One important aspect about countable groups with proper left invariant metric 
is that the $s$-scale connected components corresponds with finitely generated subgroups. This is shown in next lemma.
\begin{Lem}\label{ComponentsInCountableGroups}
Let $G$ be a countable group with $d_G$ a proper left invariant metric. Then for every $s >0$ the $s$-scale connected component $H$ of $G$ at the origin $1_G$ is a finitely 
generated subgroup of $G$. 
\end{Lem}
\begin{proof}
Suppose $s>0$ given and let $H$ be the $s$-scale connected component at $1_G$. 
As $d_G$ is proper(and discrete) the set $A = \{a_i | a_i \in G \text{ and } \|a_i\|_d \le s\}$ is finite. We will show that 
$H$ is the subgroup generated by $A$.  
Let $g\in H$ then there is a sequence $1= g_0, g_1, ..., g_n = g$ such that $d(g_i, g_{i+1}) \le s$. Hence $\|g_i^{-1}\cdot g_{i+1}\| \le s$ and
$g_i^{-1}\cdot g_{i+1}\in A$.  As $g = g_0^{-1}\cdot g_1 \cdot g_1^{-1}\cdot g_2\cdot...\cdot g_{n-1}^{-1}\cdot g_n$ we have that all elements of 
$H$ are generated by $A$.
\end{proof}

\begin{Cor}\label{FinitelyGeneratedAreConnected}
A countable group $G$ is finitely generated if and only if it is large scale connected for some (and hence for all) proper left invariant metrics $d_G$.
\end{Cor}
\begin{proof}
If $G$ is a finitely generated group then its word metric $d_G$ is a proper left invariant metric and $(G, d_G)$ is large-scale connected. It is known that 
large-scale connectedness is a coarse property and all proper left invariant metrics are coarsely equivalent (see \cite{Smith}). So we get one implication.
For the converse suppose $(G, d_G)$ is  a countable group with $d_G$ a proper left invariant metric. 
Suppose $d_G$ is $s$-scale connected. By lemma \ref{ComponentsInCountableGroups} the $s$-scale connected component 
of $G$(and hence the whole group) is a finitely generated subgroup $H$ of $G$.  
\end{proof}
In view of lemma \ref{ComponentsInCountableGroups} we can expect that properties related with the $s$-scale components in a countable group can be estimated 
just viewing the behaviour of its finitely generated subgroups. One of such properties is the asymptotic dimension.

\begin{Def} A metric space $(X, d_X)$ is said to be of
{\it asymptotic dimension} at most $n$ (notation $\asdim(X, d) \le
n$) if there is an increasing function $D_X: \RR_+ \to \RR_+$ such
that for all $s> 0$ there is a cover $\UU =\{\UU_0, ...,\UU_n\}$
so that the $s$-scale connected components of each $\UU_i$ are
$D_X(s)$-bounded i.e. the diameter of such components is bounded
by $D_X(s)$.
\end{Def}
Next theorem shows the mentioned estimation: 
\begin{Thm}(Dranishnikov and Smith \cite{Dran-Smith})\label{AsymptoticDimensionCountable} Let $G$ be a countable group then: 
\[\asdim (G )= sup\{asdim( F)\}\]
where the supremum varies over all finitely generated subgroups $F$ of $G$.
\end{Thm}

The definition of growth we will use here 
is just the application of the ideas of \cite{Roe} about the growth of general coarse structures to the particular case of countable abelian groups. 
Recall firstly that 
given two functions $g, f: \RR_+ \to \RR_+$ it is said that {\it $g$ is of type at most $f$} 
(notation: $g \preceq f$) if there exists a constant $C>0$ such that $g(x) \le C \cdot f(C\cdot x)$ with $x$ sufficiently large. If $g \preceq f$ and $f\preceq g$ it is said that 
$f$ and $g$ are of the same type.

\begin{Def}
Let $(G, d_G)$ be a countable abelian group with a proper left invariant metric, let $g\in G$, $s\in \RR_+$ and $n\in \NN$.  
We will say that a subset $A_g^{(n, s)}$ of $G$ is  
the $(n,s)$-connected component at $g$ if for every element  $h \in A_g^{(n,s)}$ there exists and $s$-scale chain of lenght $n$ that connects $g$ with $h$.\par
The {\it growth function} of $s$ at $g\in G$ is the function:
\[ gr_{(s, g)}: n \to \sharp(A_g^{(n, s)}).\]
Where $\sharp(A_g^{(n, s)})$ denotes de the cardinal of $A_g^{(n, s)}$.\par
Let $f: \NN \to \NN$ be a function. We will say that a countable group $(G, d_G)$ is of {\it growth type} at most $f$ if for every $g$ and $s$ as above 
$gr_{(s, g)} \preceq f$. Moreover 
if there exists an $g \in G$ and an $s \in \RR_+$ such that $f \preceq gr_{(s, g)}$ then we will say that the growth type of $G$ is $f$.  
\end{Def}
We need to show that in the case of finitely generated groups with word metrics then this definition coincides with the classical definition of growth 
i.e. the classical growth of a finitely generated group $(G, d_G)$ with $d_G$ a word metric is the type of the function $f: n \to \sharp(B(1_G, n))$. We will call 
this growth temporaliry classical-growth.  
\begin{Prop}\label{GeneralPropertiesGrowth}
Let $f : \NN \to \NN$ be a function, let $(G, d_{G})$ be a countable group of growth type at most $f$ and let $(H, d_H)$ be another countable group. Then:
\begin{enumerate}
\item Every subset $A$ of  $G$ is of growth type at most $f$.
\item If there exists a coarse embedding $g: H \to G$ then $H$ is of growth type at most $f$.
\item If $G$ is a finitely generated group then $G$ is of classical-growth at most $f$. 
Moreover the growth type exists for $G$ and coincides with 
the classical-growth. 
\end{enumerate}
\end{Prop}
\begin{proof}
(1) is obvious.\par
To prove (2) it is enough to prove when $g$ is inyective. The reason is that for any coarse embedding $g: H\to G$ there exists a subset $F$ of $H$ such that 
$g|_F: F \to G$ is a coarse inyective embedding and every point of $H$ is at bounded distance of $F$. Let $gr_{s, h}$ be the growth function of $s$ at $h \in H$. 
As $g$ is coarse there exists a $\epsilon_s$ such that if $d_H(h_1, h_2) \le s$ then $d_G(g(h_1), g(h_2))\le \epsilon_s$. 
Hence $\sharp(A_h^{(n,s)}) \le \sharp(A_{g(h)}^{(n, \epsilon_s)})$ what implies $gr_{s, h} \preceq gr_{\epsilon_s, g(h)} \preceq f$. \par
To prove (3) notice that $(G, d_G)$ is coarsely equivalent to $(G, d)$ with $d$ some word metric (see \cite{Smith}) by (2) the growth of $(G, d)$ is at most $f$ 
but as $(G, d)$ is quasi-geodesic then our defnition of growth coicides with the definition of classical-growth. To see it we only need that if $d(g_1, g_2) < n\cdot s$ with $g_1$ 
and $g_2$ any two elements of $G$ then 
 there exists an $s$-scale chain of length $n$. This is obvious by the definition of word metric. 
\end{proof}

\begin{Thm}\label{GrowthOfSubgroups}
Let $(G, d_G)$ a countable group with a proper left invariant metric $d_G$ and let $f: \NN \to \NN$ be a function. In this situation if the growth type of each finitely generated 
subgroup $H$ of $G$ is at most $f$ then the growth type of $G$ is at most $f$. 
\end{Thm}  
\begin{proof}
Fix $x\in G$ and $s \in \RR_+$. Notice that as $d_G$ is left invariant then $gr_{(s, x)} = gr_{(s, 1_G)}$ and the function $gr_{(s, 1_G)}$ coincides with its restriction to the
$s$-scale connected component at $1_G$. By lemma \ref{ComponentsInCountableGroups} such component coincides with the subgroup $H$ of $G$ finitely generated by the
 set $A = \{ g| \|g\|_G \le s\}$. Therefore $gr_{(s, 1_G)}$ will be bounded by the growth type of $H$.  
\end{proof}

\begin{Ex}
Let $G = \bigoplus_{i=1}^{\infty}\ZZ$ with some proper left invariant metric $d_G$ defined. Then  
for any subexponential function $f: \NN \to \NN$ the growth type of $G$ is at most $f$. This is clear as the growth type of each finitely generated subgroup is polynomial. But 
for any polynomial function $p$ there is a finitely generated subgroup $H$ of $G$ such that the growth type $q$ of $H$ is not bounded by $p$. Hence there is no function $f$ such that 
$G$ is of growth type $f$.  
\end{Ex}

\begin{Cor} A countable group $G$ is locally finite if and only if $G$ is of flat growth i.e. a constant function is a growth function.
\end{Cor}

\begin{proof} Let $f(x) = 1$ be a constant function. If $G$ is of growth type at most $f$ then all its finitely generated groups are of growth type at most $f$. 
Hence they are finite. \par
Conversely if $G$ is locally finite then for each finitely generated subgroup there is a constant $C >0$ such that $H$ is of growth type at most $f'(x) = C$. As all (non zero) 
constant function are of the same type we obtain that $G$ is of growth type at most $f'$.
\end{proof}

\begin{Cor} If a countable group $(G, d_G)$ is of polynomial growth then it is locally virtually nilpotent and of finite asymptotic dimension. 
Conversely if every subgroup $H \subset G$ is virtually nilpotent of polynomial growth with bounded degree then 
$(G, d_G)$ is of polynomial growth and finite asymptotic dimension.
\end{Cor}
\begin{proof}
By the well known Gromov's theorem (see\cite{Gromov2}) any finitely generated group of polynomial growth is virtually nilpotent. 
By Bass theorem (\cite{Bass}) we know that if  $G'$ is (finitely generated) and nilpotent and $G' = G_1 \supseteq G_2 = [G', G'] \supseteq ...\supseteq G_k = \{1\}$ 
is its lower central series then the degree $d(G')$ of the growth type of $G'$ is $d(G') = \Sigma_{i=1}^k i\cdot d_i$ with $d_i$ the torsion free rank of 
$G_n/G_{n+1}$. But by \cite{Dran-Smith} we get that the asymptotic dimension of  a nilpotent group $G'$ is $\asdim(G') = \Sigma_{i= 1}^n d_i$. The corollary follows easily
applying theorem \ref{AsymptoticDimensionCountable}, $1$ of proposition \ref{GeneralPropertiesGrowth} and theorem \ref{GrowthOfSubgroups}  
\end{proof}

\section{Countable abelian torsion groups}\label{Torsion}
The aim of this section is to study abelian groups of asymptotic dimension zero i.e. abelian torsion groups. In \cite{Brod-Dydak-Higes-Mitra} it was shown how to 
define in a locally finite group a proper left invariant metric. Let us recall the way of doing this. Such construction will become relevant to us in the following sections. 
\begin{Def}
Let $G$ be a  (countable) locally finite group. We will say that an ascending chain $\LL = \{\{1\} = G_0 < G_1 < G_2 ...\}$ of subgroups of $G$ with
$G = \bigcup_{i=1}^{\infty} G_i$   is a {\it{one-step ascending chain}}
if there exists a sequence $\{g_i\}_{i \in \NN}$ of elements of $G$ such that $G_i$ is generated by $\{g_j\}_{j=1}^i$ and  $g_{i+1} \not\in G_i$. \par
A sequence $\{g_i\}_{i= 1}^{\infty}$ with such properties will be called  {\it{one-step sequence of generators}} or simply a sequence of generators. \par
For each one-step ascending chain we have the sequence of numbers
$\{m_i = [G_{i+1} : G_i]\}_{i \in \NN \cup \{0\}}$. We will call such sequence as the {\it{one-step sequence of indexes}} or simply the sequence of indexes.
\end{Def}

\begin{Def}\label{OnUltrametrics}
Given a one-step ascending chain $\LL =\{G_0 < G_1<...\}$ in a locally finite group $G$ and a sequence of positive
numbers $\{K_i\}_{i\in \NN}$ with $K_{i+1} > K_i$ and $\lim_{i\to \infty} K_i = \infty$ we can construct the following proper left invariant ultrametric:
\[d_G(x, y) = K_i \text{ if } x^{-1}\cdot y \in G_i\setminus G_{i-1}\]
We will call such metric the ultrametric generated by $\LL$ and $\{K_i\}_{i\in \NN}$. See \cite{Brod-Dydak-Higes-Mitra} for more details.
\end{Def}

\begin{Lem}\label{AbelianCanHavePrime}
For every locally finite group $G$ there exists a one step ascending chain. If $G$ is abelian we can take a one-step ascending chain with prime sequence of indexes.
\end{Lem}
\begin{proof}
The first statement is trivial. For the second suppose given a one
step chain $\LL = \{G_0 < G_1<...\}$ with sequence of generators
$\{g_m\}_{m\in \NN}$. Suppose that all the indexes are prime till
some $G_i$. Note that $G_{i+1}/G_i$ is a cyclic group generated by
$g_{i+1}+ G_i$ with order $m_i = [G_{i+1}: G_i]$ then if $m_i =
\Pi_{j= 1}^n p_j$ with $p_j$ prime numbers (not necessarily
distinct) we will modify the sequence of generators in the
following way. The new sequence is given by $\{g_m'\}_{m\in\NN}$
with $g_m' = g_m$ if $m \le i$, $g_m' = \Pi_{j=1}^k\cdot p_j \cdot
g_{i+1}$ if $m = i+k$ with $k \le n$ and $g_m' = g_{i+s}$ if $m =
i+n+s-1$ with $s \ge 2$. Then the new one-step chain has prime
indexes till at least $G_{i+2}'$ and we can repeat the process.
\end{proof}

Next theorem is an easy consequence of some of the results of section 5 in \cite{Brod-Dydak-Higes-Mitra}(see also \cite{Prota}).

\begin{Thm}\label{OldTheorem}
A locally finite group $G$ admits an proper left invariant ultrametric $d_u$ such that $(G, d_u)$ is coarsely equivalent
to $(\bigoplus_{i=1}^{\infty} \ZZ_{p_i}, d_L)$ with
$\{p_i\}_{i\in \NN}$ a sequence of prime numbers and $d_L$ a proper left invariant ultrametric. If $G$ is abelian then we can take  $d_u$ such that
$(G,d_u)$ is isometric to $(\bigoplus_{i=1}^{\infty} \ZZ_{p_i}, d_L)$.
\end{Thm}

\begin{proof}
Given a one-step chain $\LL$ with indexes $\{a_i\}_{i\in \NN}$ we can generate an ultrametric $d_u$ as in \ref{OnUltrametrics} with the sequence $K_i =i$.
It can be checked easily that it is isometric to $G' = (\bigoplus_{i=1}^{\infty}\ZZ_{a_i}, d_L')$ (see Theorem 5.3 of \cite{Brod-Dydak-Higes-Mitra}) 
but as this group is abelian by \ref{AbelianCanHavePrime}
taking a one-step ascending
chain $\LL'$
with prime sequence of indexes $\{p_i\}_{i\in \NN}$
we generate an ultrametric isometric to $(\bigoplus_{i=1}^{\infty} \ZZ_{p_i}, d_L)$. Thus $(G, d_u)$ is bi-uniformly equivalent to $(\bigoplus_{i=1}^{\infty} \ZZ_{p_i}, d_L)$.
If $G$ is abelian we can take directly $\{a_i\}_{i\in\NN}$ as prime numbers in the first step.
\end{proof}
We focus now on the coarse classification of abelian torsion groups or more generally in view of \ref{OldTheorem} in locally finite groups. 
Such classification corresponds with problem 5.3. of \cite{Brod-Dydak-Higes-Mitra} (mentioned as problem 1606 in the book \cite{Pearl}). Now 
by Theorem \ref{OldTheorem} we can build a proper left invariant metric in a locally finite group coarsely equivalent to $(\bigoplus_{i=1}^{\infty}\ZZ_{p_i}, d_L)$ with 
$d_L$ an ultrametric. So we need a coarse classification of groups $(\bigoplus_{i=1}^{\infty}\ZZ_{p_i}, d_L)$ with $d_L$ an proper ultrametric. \par
The crucial step to do this 
has appeared recently in a paper of Banakh and Zarichnyi \cite{Ban-Zar}. In such paper the authors showed that all groups of the form 
$(\bigoplus_{i=1}^{\infty}\ZZ_{p_i}, d_L)$ are the base $[T]$ of an asymptotically homogeneous tower
(see the paper for more details) and also they proved the following  
theorem(Theorem 4 of \cite{Ban-Zar}):

\begin{Thm}(Banakh-Zarichnyi \cite{Ban-Zar})The base of each asymptotically homogeneous tower $T$ is asymorphic(coarsely equivalent) the anti-Cantor set $2^{<\omega}$. 
\end{Thm}

A direct consequence of this theorem (and the previous theorem \ref{OldTheorem}, see \cite{Brod-Dydak-Higes-Mitra}) is 
the following corollary that solves completely the cited problem:

\begin{Cor}\label{TorsionAbelian} Any two locally finite groups are coarsely equivalent.
\end{Cor}

\section{Odd pairs and coarse splitting}\label{Odd}

The notion of one-step ascending chain can be generalized to
countable groups. Let $H$ be a normal subgroup of a countable
group $G$ such that $G/H$ is locally finite. Let $\pi: G \to G/H$
be the usual projection. Now take a one-step ascending chain $\LL'
= \{H/H = G'_0 < G'_1<...\}$ of $G/H$ with $\{g'_i\cdot H\}_{i \in
\NN}$ a sequence of generators and $\{m_i\}$ the sequence of
indexes. The ascending chain $\LL = \{H = G_0 < G_1 <...\}$ of $G$
given by $G_i = \pi^{-1}(G'_i)$ will be called a {\it{one-step
ascending chain associated to $H$}} or a {\it{one-step ascending
chain derived from  $\LL'$}} when the ascending chain $\LL'$ needs
to be remarked. A sequence $\{g_i\}_{i\in \NN}$ such that $g_i \in
g'_i \cdot H$ will be called a sequence of generators and the
sequence $\{m_i\}_{i\in\NN}$ will be the sequence of indexes. \par
Given $d_H$ and $d_{G/H}$ two proper left invariant metrics 
we are
interested in the process of building a (nice) proper left
invariant metric $d_G$ related somehow with $d_H$ and $d_{G/H}$. We will show that it is possible
for some kind of pairs $(G, H)$ of countable abelian groups $H <
G$ and for some kind of metrics $d_{G/H}$.

\begin{Def} We will say that a pair of countable abelian groups $(G, H)$ with $H < G$ is an {\it{odd pair}} if every element of $G/H$ has odd order.\par
Analogously a pair of countable abelian groups $(G, H)$ with $H < G$ will  be an {\it even pair} if the order of every element of $G/H$ is a power of $2$.
\end{Def}
Elements of an odd pairs has the following nice representation.

\begin{Lem}
\label{Representation} Let $(G, H)$ be an odd pair and let $\LL
=\{G_0 < G_1...\}$ be a one-step ascending chain associated to $H$
with generators $\{g_i\}_{i\in \NN}$ and (odd) indexes $\{m_i =
2\cdot k_i +1\}_{i \in \NN}$ then every element $g \in G$ can be
written uniquely as:
\begin{equation*}
g = h_g + \sum_{i=1}^{\infty} r_i\cdot g_i \text{ (*) }
\end{equation*}
with the properties:
\begin{enumerate}
\item $h_g \in H$
\item $r_i \in [-k_i , k_i] \cap \ZZ$.
\item There exists a minimum $i_0$ so that $r_i = 0$ for every $i > i_0$.
\end{enumerate}
\end{Lem}

\begin{proof}
To prove this we note that $G_i /G_{i-1} = \{-k_i\cdot g_i + G_{i-1}, ..., -g_i+G_{i-1}, G_{i-1}, g_i+ G_{i-1},..., k_i\cdot g_i+G_{i-1}\}$. For each $g \in G_i$ we define
$r_i(g)$ as the
integer  of $[-k_i, k_i]\cap \ZZ$ such that $g$ is in the class $r_i(g) \cdot g_i+G_{i-1}$.
Given $g \in G$ let $n$ be the minimum integer such that $g \in G_n$. We define the numbers $r_i$ and
the element $h_g\in H$ in the following way:
\[r_i = 0 \text{ if } i> n\] \[r_n = r_n(g)\] \[r_{i-1} = r_{i-1}(g- \sum_{k =i}^n r_k \cdot g_k) \text{ if } i>1 \] \[h_g = g-\sum_{k = 1}^n r_k \cdot g_k\]
It is obvious that $g = h_g + \sum_{i=1}^{\infty} r_i \cdot g_i$ and that $r_i$ and $h_g$ verify the three properties.
Finally uniqueness is an easy consequence of the construction.
\end{proof}
This representation theorem is the main difference between even pairs and odd pairs. For a non odd pair there will be exists some elements in $G/H$ of order $2$, it means 
elements equal to its inverse and as we see we would run into problems when we construct $d_G$.\par
 
The following construction for odd pairs is similar to the one of theorem \ref{OldTheorem}
\begin{Lem}\label{NormsAreNorms} Let $(G, H)$ be an odd pair and let $\LL =\{H = G_0 < G_1 <...\}$ be an ascending chain derived from 
$\LL' = \{G'_0< G'_1<...\}$. Suppose there exist a proper norm $\|\cdot\|_H$ and a sequence of positive numbers 
$\{K_n\}_{n\in \NN}$ such that for every $n \in \NN$ we have $K_n \ge \sum_{i= 1}^{n} (n-i+1)\cdot\| h_{m_i \cdot g_i}\|_H$ and $K_{n+1} > K_n$. If we define the proper left invariant ultrametric in $G/H$ associated to $\LL'$ and $\{K_n\}$ (see \ref{OnUltrametrics} then  the function $\|\cdot\|_G: G \times G \to \RR_+$ defined by:
\[\|g = h_g + \sum_{i=1}^{\infty} r_i\cdot g_i\|_G = \|h_g\|_H + \|\pi(\sum_{i=1}^{\infty} r_i\cdot g_i)\|_{G/H}\]
is a proper norm in $G$.
\end{Lem}
\begin{proof}
We have to check the three conditions of a norm plus the proper condition. Clearly $\|g\|_G = 0$ if and only if $g = 0$ and by Lemma \ref{Representation} we obtain
$-g = -h_g + \sum_{i=1}^{\infty} (-r_i)\cdot g_i$ and as $\|\cdot\|_H$ and $\|\cdot\|_{G/H}$ are norms then $\|-g\|_G = \|g\|_G$.
We will prove the subadditivity by induction showing that the restriction of $\|\cdot\|_G$ to $G_n$ satisfies
$\|x+y\|_{G_n} \le \|x\|_{G_n} + \|y\|_{G_n}$. The case $n = 0$ is trivial as $\|\dot\|_H$ is a norm.
Now suppose the subadditivity true for every group $G_i$ with $i <n$ and let $x, y\in G_n$.
Assume $x \in G_n \setminus G_{n-1}$. By Lemma \ref{Representation} we get:
\[x = h_x + \sum_{i = 1}^n r_i \cdot g_i\] with $r_n \ne 0$
\[y = h_y + \sum_{i=1}^n r'_i \cdot g_i.\]
Let $j$ be the maximum $i$ such that $r'_i \ne 0$. If $j = n$ and $r'_j = -r_n$ the result is a consequence of the induction hypothesis. 
So we assume $j< n$ or $r_j'\ne -r_n$ then:
\[x + y = h_x + h_y + \sum_{i = 1}^j (r_i + r'_i) \cdot g_i + \sum_{i = j+1}^n r_i\cdot g_i\]
With the last term possibly $0$ by the convention $\sum_{i=k}^m x_i = 0$ if $k >m$. \par
Let $h_z$ be the element of $H$ from the representation of $z = \sum_{i = 1}^j (r_i + r'_i) \cdot g_i$. Notice that
for $h_z$ there exists $\{s_i\}_{i= 1}^j$ such that $h_z$ is of the form $h_z = \sum_{i = 1}^j s_i \cdot h_{m_i \cdot g_i}$ and $|s_i| < j-i+1$. It yields to:
\[\|x+ y\|_G = \|h_x + h_y + h_z\|_H + K_n \le \|h_x\|_H + \|h_y\|_H + \|h_z\|_H + K_n \le \]
\[\le \|h_x\|_H + \|h_y\|_H + \sum_{i= 1}^{j} (j-i+1)\cdot\| h_{m_i \cdot g_i}\|_H + K_n \le  \|h_x\|_H + \|h_y\|_H + K_j + K_n\]
The last term is  equal to $\|x\|_G + \|y\|_G$. So it is a norm. \par
Properness is an easy consequence of the fact that the norms $\|\cdot\|_H$ and $\|\cdot\|_{G/H}$ are proper.
\end{proof}

The corresponding metric $d_G$ of the previous lemma will be called the {\it{pseudo-ultrametric generated by $(d_H, \LL)$}}.

\begin{Lem}\label{OddPairsAreBiuniform} Let $(G, H)$, $(G', H')$ be two odd pairs such that
$H$ and $H'$ are isomorphic. Let $\|\cdot\|_H$ be a proper norm of
$H$ and let $\|\cdot\|_{H'}$ be the induced norm in $H'$ via the
isomorphism. If there exist two one step ascending chains $ \LL =\{G_0 < G_1
<...\}$ and $\LL'=\{G'_0 < G'_1<...\}$ associated to $H$ and $H'$
respectively such that the sequence of indexes are equal then $(G,
d_G)$ and $(G', d_{G'})$ are coarsely equivalent with $d_G$
and $d_{G'}$ the pseudo-ultramentrics generated by  $(d_H, \LL)$
and $(d_{H'}, \LL')$ respectively. Moreover the coarse equivalence can be chosen to be biyective. 
\end{Lem}
\begin{proof}
Let  $r: H \to H'$ be an isomorphism between $H$ and $H'$.
We will use Lemma \ref{Representation}. Let $\{g_i\}_{i\in \NN}$ and $\{g'_i\}_{i\in \NN}$ be two sequence of generators of $\LL$ and
$\LL'$ and let $\{K_n\}_{n\in\NN}$, $\{K'_n\}_{n\in\NN}$ be the sequences of numbers used to generate the ultrametrics $d_{G/H}$ and $d_{G'/H}$.
Given $g\in G$ with $g = h_g + \sum_{i = 1}^{\infty} r_i \cdot g_i$ we build the map $f: G \to G'$ defined by $f(g) = r(h_g) + \sum_{i= 1}^{\infty} r_i \cdot g'_i$.
It is clearly bijective. Let us show it is a coarse map.  \par Given $K >0$.
we have to prove that there exist a $C_K$ such that $d_G(x, y) \le K$ implies  $d_{G'}(f(x),
f(y)) \le C_K$ for every $x, y \in G$. Firstly let $i_K$ be the maximum index $i$
such that $K_i\le K$. Let $\{h_i\}_{i=1}^{i_K}$ be the finite set of elements of $G$ with $h_i = h_{m_i\cdot g_i}$ and let $\{h'_i\}_{i=1}^{i_K}$ be the analoguos set
in $G'$ i.e. $h'_i = h'_{m_i\cdot g'_i}$. Define the constant:
\[C_K = K + K'_{i_K}+ \sum_{i = 1}^{i_K} (i_K-i+1) \cdot \|h'_i\|_{H'}+ \sum_{i = 1}^{i_K} (i_K-i+1) \cdot \|h_i\|_H\]
Let us show that this constant satisfies the coarse conditon.
 Given
$x = h_x + \sum_{i=1}^{\infty} p_i \cdot g_i$ and $y = h_y +
\sum_{i=1}^{\infty} q_i \cdot g_i$ with \[d_G(x, y) = \|h_y-
h_x + \sum_{i=1}^{\infty} (q_i-p_i) \cdot g_i\|_G \le K\]
There exists a minimum index $i_0$ such that $q_i-p_i = 0$
for every $i>i_0$. Notice $i_0 \le i_K$. As in the previous proof we have that there is some $h_z$ such that:
\[d_G(x,y) = \|h_y - h_x + h_z\|_H + K_{i_0}\le K \text{, and }  h_z = \sum_{i = 1}^{i_0} s_i \cdot h_i \text{ with } |s_i| < i_0-i+1. \]
Analogously in $G'$:
\[d_{G'}(f(x), f(y)) = \|r(h_y)-r(h_x)+ h'_z\|_{H'} + K'_{i_0} \text{,  } h'_z = \sum_{i = 1}^{i_0} s'_i \cdot h'_i \text{ with }|s'_i| < i_0-i+1.\]
Now we get:
\[ \|r(h_y)-r(h_x)+ h'_z + r(h_z) -r(h_z)\|_{H'} + K'_{i_0} \le  K + \|h'_z\|_{H'} + \|r(h_z)\|_{H'}+ K'_{i_K} \le \]
\[\le K'_{i_K} +  K+ \sum_{i = 1}^{i_0} (i_0-i+1) \cdot \|h'_i\|_H'+\sum_{i = 1}^{i_0} (i_0-i+1) \cdot \|h_i\|_H \le C_K \]
And $f$ is a coarse map. By symmetry in the reasoning the inverse
$f^{-1}$ of $f$ given by $f^{-1}(h'_g + \sum_{i=1}^{\infty} r_i\cdot g'_i) =
r^{-1}(h'_g) + \sum_{i=1}^{\infty}r_i \cdot g_i$ will be a coarse
map and hence $(G, d_G)$ and $(G', d_{G'})$ will be coarsely equivalent with $f$ a biyection.
\end{proof}
Next definition suggests an coarse analogous to the algebraic notion of splitting in exact sequences.

\begin{Def}\label{CoarselySplit}
Let $K, G, H$ be groups. Given an exact sequence by:
\[1\to
K \overset{i}{\rightarrow} G \overset{\pi}{\to} H\to 1\] We will
say that such sequence is {\it{coarsely split}} if there exists proper left invariant metrics $d_K, d_G$ and $d_H$ and a coarse
equivalence $f: (G, d_G) \to (K\oplus H, d_K \oplus d_H)$ such
that the distance between $f\circ i$  and $id|_H$ and the distance
between $\pi' \circ f$ and $\pi$ are bounded with $\pi': K\oplus H
\to H$ defined as the natural projection. The metric $d_K\oplus d_H$ is defined as the usual sum metric i.e. the $l_1$-metric of a product of two spaces.
\end{Def}

\begin{Thm}\label{OddPairsSplitsCoarsely}
Let $(G, H)$ be an odd pair.  The exact sequence:
\[0\to
H \overset{i}{\rightarrow} G \overset{\pi}{\to} G/H\to 0\] is
coarsely split.
\end{Thm}
\begin{proof}
Let $d_H$ be any proper left invariant metric of $H$ and let $\LL$ be a one-step ascending chain of $G/H$.
Take in $d_G$ the pseudo-ultrametric generated by $d_H$ and $\LL$. Suppose that $\{K_i\}_{i \in \NN}$ are the numbers used to generate $d_G$ as in \ref{NormsAreNorms}.
Let $d_{G/H}$ be a  metric of  $G/H$  generated by $\LL$ and $\{K_i\}_{i\in \NN}$. By \ref{OddPairsAreBiuniform} we get that the map 
$f: (G, d_G) \to (H\oplus G/H, d_H \oplus d_{G/H})$ defined by $f(g) = (h_g, \pi(g))$ is a coarse equivalence that satisfies $f\circ i = id|_{H}$ and $\pi'\circ f = \pi$.
\end{proof}
\begin{Rem} The metric $d_{G/H}$ is an ultrametric.
\end{Rem}
\section{Even pairs}\label{Even}
In this section we study coarse splitting of even pairs. The most easy case of even pair is $(\Zmed, \Ztwo)$ with the groups defined as 
$\Zmed = \{\frac{m}{2^k} \text{,  }m \in \ZZ \text{ and } 
k\in \NN\cup \{0\} \}$ and $\ZZ(2^{\infty})$  the direct limit of $\ZZ_2 \to \ZZ_4 \to ...\to \ZZ_{2^i}\to...$. It is clear that $\Ztwo$ is isomorphic to 
$\Zmed/\ZZ$.

\begin{Lem}\label{ZmedSplitsCoarsely} The exact sequence:
\[0 \to \ZZ \to \Zmed \to \ZZ(2^{\infty}) \to 0\]
is coarsely split. Moreover $d_{\ZZ}$ can be chosen as the natural word metric.
\end{Lem}

 \begin{proof}
Each element $g\in \Zmed\setminus \ZZ$ can be written uniquely as:
\[g = m_g + \frac{k_g}{2^{i_g}}\] with $i_g >1$ and $m_g \ge 0$,  $0 < k_g < 2^{i_g}$  if $g
\ge 0$ or $m_g < 0$, $-2^{i_g} < k_g < 0$ if $g<0$ and
$(k_g, 2^{i_g}) = 1$. Take the norm $\|\cdot\|_{\Zmed}$ defined
as:
\[\|g\|_{\Zmed} = |g| \text{ if } g\in \ZZ\]
\[\|g\|_{\Zmed} = |m_g| + i_g \text{ if } g \in \Zmed\setminus\ZZ \]
Let us check that it is a proper norm. Clearly it is proper and satisfies
$\|g\|_{\Zmed} = \|-g\|_{\Zmed}$ and $\|g\|_{\Zmed} = 0$ if and
only if $g =0$.. The unique
remainder property is $\|x+y\|_{\Zmed} \le \|x\|_{\Zmed} +
\|y\|_{\Zmed}$ for every $x, y \in \Zmed$. We will write $g = m_g +
\frac{k_g}{2^{i_g}}$ with $k_g = 0 $ if $g \in \ZZ$ and $i_g$ any
positive integer. Define $\delta_{x,y}$ as:
\[\delta_{x,y} =  \begin{cases} -1 \text{ if } (m_x + m_y > 0 \text{ and } 
\frac{k_x}{2^{i_x}} + \frac{k_y}{2^{i_y}} < 0) \text{ or }  (m_x + m_y < 0 \text{ and } 
\frac{k_x}{2^{i_x}} + \frac{k_y}{2^{i_y}} \le -1)\\ 0 \text{ if } (m_x + m_y \ge 0 \text{ and }1 > \frac{k_x}{2^{i_x}} + \frac{k_y}{2^{i_y}} \ge 0) \text{ or } 
( m_x + m_y \le 0 \text{ and } -1 <  \frac{k_x}{2^{i_x}} + \frac{k_y}{2^{i_y}} \le 0)\\ 1 \text { otherwise }\end{cases}   \]
In such situation for every $x, y \in
\Zmed$ we also define $i_{x,y} = \max\{|sign(k_x)|\cdot i_x, |sign(k_y)|\cdot
i_y\}$. Hence we have:
\[\|m_x + \frac{k_x}{2^{i_x}} + m_y +
\frac{k_y}{2^{i_y}}\|_{\Zmed} = |m_x + m_y + \delta_{x,y}| +
i_{x,y}\] Notice that if $k_x$ and  $k_y$ are not zero  then:
\[i_{x,y} + |\delta_{x,y}| \le i_x + i_y.\]
So the unique non trivial case is when $k_x = 0$ and $ 0 \ne sign(m_x + m_y) \ne sign(\frac{k_y}{2^{i_y}})$. 
But in such case we have $|m_x + m_y+ \delta_{x, y}| \le |m_x + m_y|$ and this implies $\|\cdot\|_{\Zmed}$ is a norm.
Now let $d_{\Ztwo}$ the ultrametric generated by the one-step ascending chain $ 1 < \ZZ_2 < \ZZ_4 <...\ZZ_{2^i}< ...$ 
and the sequence $\{i\}_{i \in \NN}$ (see remmark \ref{OnUltrametrics}). Define the surjective
map $f: (\Zmed, d_{\Zmed}) \to (\ZZ \oplus \Ztwo, D:= d_{\ZZ}\oplus d_{\Ztwo})$ as $f(x) = (m_x, \pi(x))$.   
We claim that $f$ is a coarse equivalence. Firstly we will prove it is a coarse map. 
Given $K >0$ and $x, y \in \Zmed$ suppose $\|x-y\|_{\Zmed} \le K$. Define $\delta_{x, -y}$ and $i_{x,-y}$ as above then:
\[\|x-y\|_{\Zmed} = |m_x-m_y+\delta_{x, -y}| + i_{x,-y} \le K\]
Notice that $i_{x, -y}$ is in fact $\|\pi(x)-\pi(y)\|_{\Ztwo}$ i.e. the minimum positive integer such that $\pi(x)-\pi(y) \in \Ztwo\setminus\ZZ_{2^{i_{x,-y}-1}}$ then:
\[D(f(x), f(y)) = |m_x-m_y| + i_{x, -y}\le K - |\delta_{x, -y}|\]
And the map is coarse. \par
Given $x \in \Ztwo$ we define $i(x) = \|x\|_{\Ztwo}$ then there exists an unique 
$k'(x) \in \{0, 1,..., 2^{i(x)}-1$ such that 
the class $[k'(x)]$ of $k'(x)$ in $\Ztwo$ is $x$. Take $g:  (\ZZ \oplus \Ztwo, D) \to  (\Zmed, d_{\Zmed}$ as the map:
\[g((m, x) = \begin{cases} m + \frac{k(x)}{2^{i(x)}} \text{ if } m \ge 0\\ m + \frac{-2^{i(x)}+ k(x)}{2^{i(x)}} \text{ if } m < 0 \end{cases}\]
This map is injective and applying a similar reasoning as before we get it is  coarse. 
Also we have $f\circ g = id$ and if we restrict $f$ to $g(\ZZ \oplus \Ztwo)$ we obtain $g\circ f = id$. 
The unique points not in the image of $g$ are those $x$ such that $x \in (-1, 0)$. Let $x = -\frac{k_x}{2^{i_x}}$ be one of such points. 
Take the element $ y = \frac{2^{i_x}- k_x}{2^{i_x}}$ that is in the image of $g$ and we get $d_{\Zmed} (x, y) = \|x-y\|_{\Zmed} = 1$. 
Hence $g$ is a coarse inverse of $f$. Finally by contruction we get that $f\circ i = id|_{\ZZ}$ and $\pi'\circ f = \pi$.
\end{proof}

Combining this lemma and theorem \ref{OddPairsSplitsCoarsely} we get the following example that shows the coarse geometry of $\QQ$. In \cite{Smith} it was asked about 
the geometric structure of $\QQ$. 

\begin{Ex} $\QQ$ is coarsely equivalent to $\ZZ \oplus \QQ/\ZZ$. The reason is the following. 
Firstly we have $(\QQ, \Zmed)$ is an odd pair. Hence $\QQ$ is coarsely equivalent 
to $\Zmed \oplus \QQ/\Zmed$. Now by previous lemma $\Zmed$ is coarsely equivalent to $\ZZ \oplus \Ztwo$. Finally we have that $(\QQ/\ZZ, \Ztwo)$ is and odd pair and then 
$\Ztwo \oplus \QQ/\ZZ$ is coarsely equivalent to $\QQ/\ZZ$. Notice that by corollary \ref{TorsionAbelian} $\QQ/\ZZ$ is coarsely equivalent to any locally finite group.
\end{Ex}

\begin{Cor}\label{ZmedPowerNSplitsCoarsely}
The exact sequence
$$0 \to \bigoplus_{i=1}^K \ZZ \overset{i}{\to} \bigoplus_{i=1}^K \Zmed \overset{\pi}{\to} \bigoplus_{i=1}^K \Ztwo \to 0$$
is coarsely split if $K$ is finite.
\end{Cor}
\begin{proof}
If $f: (\Zmed, d_{\Zmed})\to (\ZZ\oplus \Ztwo, d_{\ZZ}\oplus
d_{\Ztwo})$ is the coarse equivalence defined in
\ref{ZmedSplitsCoarsely} then as $K$ is finite is not hard to
check that the map $F: (\bigoplus_{i=1}^K\Zmed, d_{\Zmed}^n) \to
(\bigoplus_{i=1}^K (\ZZ\oplus \Ztwo), \bigoplus_{i=1}^K
(d_{\ZZ}\oplus d_{\Ztwo})$ defined as the product map $F(x_1,
...,x_k) = (f(x_1),...,f(x_k))$ is a coarse equivalence that
satisfies the conditions of the definition of \ref{CoarselySplit}.
\end{proof}
To solve the even case we will try to generalize the proccess we have done for $\QQ$ to all countable abelian groups.
For such goal we will use the notion of
$2$-divisibility and the idea that an even pair admits a nice
embedding in a direct sum of copies of $\Zmed$  and
$\ZZ(2^{\infty})$ i.e. a $2$-divisible group. 
Then the main theorem of this theorem will be deduced from the previous corollary.

\begin{Def} An abelian group $G$ is said to be $2$-divisible if the equation
 $2\cdot x = a $ has a solution $x$ for every $a \in G$.\par
We will say that an abelian  group $G$ is {\it{even-by-quotient}} if $(G, H)$ is an even pair with $H$ the subgroup generated by a
maximal free independent system.
\end{Def}
\begin{Rem}
In an even-by-quotient group the order of any element in the torsion is a power of $2$.
\end{Rem}
Next algebraic result shows that each even by quotient group admit an isomorphic embedding into a $2$-divisible group. Such result could be consider a particular 
case of a more general case: each abelian group can be embedded isomorphically in a divisible group (see \cite{Fuchs}). We will give a complete proof 
as it is not clear how to get 
the second part of next proposition from the theorem of Fuchs.\par 
Recall that the {\it socle} of a group is the subgroup in which the orders of its elements are free-square.

\begin{Prop}\label{EvenCanBeEmbedded}
Let $G$ be an even-by-quotient group and let $r_0(G)$ and $r_2(G)$ be the torsion free rank and the $2$-rank then there exists a monomorphism 
$f$ of $G$ into
$\bigoplus_{i=1}^{r_0(G)} \Zmed \oplus \bigoplus_{i= 1}^{r_2(G)}\ZZ(2^{\infty})$. 
Moreover given $L = \{g_i\}_{i=1}^{r_0(G)} \cup \{h_i\}_{i=1}^{r_2(G)}$ a maximal independent system of $G$ with $h_i$ contained 
in the socle then we can find $f$ such that $f(g_i) = (0,...0, 1, 0,..., 0,...)$ and $f(h_i) = (0,...,0, 0,..., 1, 0,...)$. \end{Prop}
In the following the torsion free rank of a group will be denoted by $r_0(G)$ or $dim_{\QQ}(G\otimes \QQ)$ indistinctly. The first notation appears in 
the classic book of \cite{Fuchs}.
\begin{proof}
Let $L = \{g_i\}_{i=1}^{r_0(G)} \cup \{h_i\}_{i=1}^{r_2(G)}$ be a
maximal independent system of $G$ with $h_i$ contained in the
socle. As it is independent we get $< L> =
\bigoplus_{i=1}^{r_0(G)} \ZZ \oplus
\bigoplus_{i=1}^{r_2(G)}\ZZ_2$. Where $<L>$ means the subgroup generated by $L$. So $<L>$ can be embedded naturally
in $\bigoplus_{i=1}^{r_0(G)} \Zmed \oplus \bigoplus_{i=
1}^{r_2(G)}\ZZ(2^{\infty})$. Let $f$ be such emebedding, we want
to extend $f$ to the whole group. Firstly notice that  as $L$ is a
maximal independent system for every $g \in G \setminus <L>$ there
exists a minimum positive integer $m >0$ such that:
\[ 0\ne m\cdot g = a_1 +a_2 \text{ with } a_1 \in \bigoplus_{i=1}^{r_0(G)} \ZZ \text{ and  } a_2 \in \bigoplus_{i=1}^{r_2(G)} \ZZ_2\]
Now if there exists a prime $p \ne 2$ such that $p | m$ then the
order of the class $g + \bigoplus_{i=1}^{r_0(G)} \ZZ$ in
$G/\bigoplus_{i=1}^{r_0(G)} \ZZ$ will be $m$ if $a_2 = 0$ or
$2\cdot m$ otherwise. But this contradicts the fact that $G$ is
even by quotient. Hence $m = 2^k$ for some positive number $k >0$.\par The
extension of $f$ will be done in two steps:\par Step 1: {\it{There
exists an injective extension $\tilde{f}$ to $H  = < L, T_2>$
where $T_2$ is the torsion subgroup of $G$}}.\par We proceed by
induction on the order $2^k$ of the elements to extend $f$. It is
clear from the preceding reasoning that all the elements of $G$ of
order $2$ are in $<L>$ and then the case $k= 1$ is true. Let
$H_{k}$ be of the form $H_k = <L, M_k>$ with $M_k$ the set of all
elements of order less or equal $2^k$. Suppose defined  a function 
$f_k: H_k \to
\bigoplus_{i=1}^{r_0(G)} \Zmed \oplus \bigoplus_{i=
1}^{r_2(G)}\ZZ(2^{\infty})$  that is a monomorphism. 
We will show that
there exists a monomorphic extension $f_{k+1}: H_{k+1} \to
\bigoplus_{i=1}^{r_0(G)} \Zmed \oplus \bigoplus_{i=
1}^{r_2(G)}\ZZ(2^{\infty})$. In fact by Zorn's lemma 
it is enough to prove that if
there exists an extension ${\bar{f}}_k$ of $f_{k}$ to the subgroup
$H_{k+1}' = < H_k \cup T_{k+1}>$ with $T_{k+1}$ any
subset(possibly empty) of elements of order $2^{k+1}$ then there
exists an extension ${\tilde{f}}_k$ of ${\bar{f}}_k$ to $
<H_{k+1}' \cup \{g\}>$ with $g$ any element of $H_{k+1}\setminus
H_{k+1}'$. Let $g$ be such element. There exists a minimum $m \le
k$ and $h_g\in H_{k+1}'$ such that:
\[ 2^m \cdot g = h_g\]
Let $n>0$ be an integer such that 
$n\cdot g = h$ for some $h\in H_{k+1}'$, we claim that $2^m | n$. Clearly if such $n$ exists with $2^m \not| n$ i
t must be of the form $n = 2^l\cdot p$ with $0 \le l < m$ and $p >1$ an odd number i.e. $p = 2^r\cdot q -1$, $r>1$. In such case we would get:

\[2^{r+l}\cdot q \cdot g = 2^l \cdot g +h \]

Notice that by the minimality of $m$ it can not happen that $r+l\ge m$ but then:

\[2^m\cdot q \cdot g = 2^{m-r}\cdot g + 2^{m-r-l}\cdot h \text{ multiplying by } 2^{m-r-l}\]

this yields $2^{m-r}\in H_{k+1}'$  a contradiction with the minimality of $m$.
Now let ${\tilde{f}}_k(g) = x$ with $x$ some solution of the
equation $2^m \cdot x = {\bar{f}}(h_g)$. As any element $y$ of
$<H_{k+1}' \cup \{g\}>$ is of the form $y = n \cdot g + h$ with $h
\in H_{k+1}'$ we define ${\tilde{f}}(y) = n\cdot x +
{\bar{f}}(h)$. By the previous reasoning ${\tilde{f}}$ is well
defined an it is an extension of $\bar{f}$ and an homomorphism. To prove that ${\tilde{f}}$ is a
monomorphism it is enough to prove that if $n\cdot x +
{\bar{f}}(h) = 0$ for some $n >0$ and $h \in H_{k+1}$ then $n\cdot
g + h = 0$. Let $n = 2^l \cdot p$ with $p$ and odd number
(possibly $1$). Suppose on the contrary there exists such $n$ and
$h$ with $n\cdot g + h \ne 0$ and $n\cdot x + {\bar{f}}(h) = 0$.
As $\bar{f}$ is a monomorphism we can assume $l <m$. If $p =1$ then from
${\bar{f}}(2^m \cdot g + 2^{m-l}\cdot h ) = 0$ and the injectivity
of $\bar{f}$ we get $2^m\cdot g + 2^{m-l}\cdot h = 0$. It yields
to $2^{m-l}\cdot(2^l\cdot g + h) = 0$ and the order of $2^l \cdot
g + h$ is less or equal $2^{m-l}$. Hence $2^l \cdot g + h \in
H_{k+1}'$ a contradiction with the minimality of $m$. But if $p
>1$ then $p = 2^r\cdot s - 1$ with $1<r$. In this case we can suppose $l+r <m$ if $l+r \ge m$ then we are in the previous case. 
By
the injectivity of $\bar{f}$ and a similar reasoning as before we get that $2^{r}\cdot( s\cdot h_g -2^{m-r}\cdot g +
2^{m-l-r}\cdot h) = 0$. It implies
that the order of $s\cdot h_g -2^{m-r}\cdot g +  2^{m-r-l}\cdot h$ is
less or equal $2^{r}$ and then $2^{m-r}\cdot g \in H_{k+1}'$
a contradiction with the minimality of $m$.

Step 2: {\it{There exists an injective extension $\tilde{F}$ of $\tilde{f}$ to the whole group  $G$}}\par
The reasoning is similar to the previous one. By Zorn's Lemma it is enough to show that if there exists a monomorphic extension ${\tilde{f}}_{\lambda}$ of ${\tilde{f}}$ to
$G_{\lambda} = <H \cup M_{\lambda}>$ with $M_{\lambda}$ any subset of $G \setminus H$ then there is a monomorphic extension ${\tilde{F}}_{\lambda}$ of
${\tilde{f}}_{\lambda}$ to $G_{\lambda}' = < G_{\lambda}\cup \{g\}>$ with $g \in G \setminus G_{\lambda}$. Suppose ${\tilde{f}}_{\lambda}$, $G_{\lambda}$ and
$g \in G\setminus G_{\lambda}$ given. Let $m >0$ be the minimum integer for which there exists a $h_g \in G_{\lambda}$ such that:
\[2^m \cdot g = h_g\]
Notice that reasoning as before if there exists an integer $n>0$ and a $h \in G_{\lambda}$ such that $n \cdot g = h$ then $2^m | n$. Let
$x \in \bigoplus_{i=1}^{r_0(G)} \Zmed \oplus \bigoplus_{i=1}^{r_2(G)}\Ztwo$ a solution of the equation $2^m \cdot x = {\tilde{f}}_{\lambda}(h_g)$. We know that any element
$g'\in G_{\lambda}'$ can be written as $ g' = n\cdot g + h$ for some integer $n$ and $h\in G_{\lambda}$ then define the map
${\tilde{F}}_{\lambda}: G_{\lambda}' \to \bigoplus_{i=1}^{r_0(G)} \Zmed \oplus \bigoplus_{i=1}^{r_2(G)}\Ztwo$ as
${\tilde{F}}_{\lambda}(g') = n\cdot x + {\tilde{f}}_{\lambda}(h)$. By construction it is well defined and it is an homomorphism.
It remains to show the injectivity. Suppose that there exists an $n >0$ such that $n\cdot x + {\tilde{f}}_{\lambda}(h) = 0$ with $n = 2^l\cdot p$ and $2^m \not | n$.
Reasoning as in the previous step we would get that or the order of $2^l\cdot g + h$ is less or equal $2^{m-l}$ if $p = 1$ or the order of
$s\cdot h_g -2^{m-r}\cdot g +  2^{m-r-l}\cdot h$ is less or equal $2^{r}$ if $p = 2^r \cdot s-1$. As every element of finite order is in $H$ then
we obtain a contradiction with the minimality of $m$ then ${\tilde{\tilde{f}}}_{\lambda}$ is a monomorphism and we have completed the proof.
\end{proof}
Next result solves the even case.

\begin{Thm}\label{EvenByQuotientSplitsCoarsely} Let $G$ be an even-by-quotient group such that $\dim_{\QQ} (G \otimes \QQ)< \infty$.
If $H$ is a subgroup generated by a maximal free independent system of elements of $G$ then the exact sequence:
\[0\to
H \overset{i}{\rightarrow} G \overset{\pi}{\to} G/H\to 0\]
is coarsely split.
\end{Thm}
\begin{proof}
By Proposition \ref{EvenCanBeEmbedded} we can assume that $G$ is a subgroup of ${\mathbb{G}} := \bigoplus_{i= 1}^{r_0(G)}\Zmed \oplus \bigoplus_{i= 1}^{r_2(G)} \Ztwo$
with $H$ the
subgroup $\bigoplus_{i=1}^{r_0(G)} \ZZ$. As an easy consequence of Corollary \ref{ZmedPowerNSplitsCoarsely} we get that the exact sequence
\[0 \to H \overset{i}{\to} {\mathbb{G}} \overset{\pi}{\to} {\mathbb G}/H \to 0\]
is coarsely split. Let $F: ({\mathbb G}, d_{\mathbb G}) \to (H \oplus {\mathbb G}/H, d_H \oplus d_{{\mathbb G}/H})$ be the coarse equivalence that splits the exact sequence. Then
the restriction $F|_G$ will be a coarse embedding of $G$ into $H \oplus {\mathbb G}/H$. In particular as the projection $\pi: {\mathbb G} \to {\mathbb G}/H$ coincides with
the composition of the projection ${\bar{\pi}}: H \oplus {\mathbb G}/H\to {\mathbb G}/H$ with $F$ we get that $F|_G(G) \subseteq H \oplus G/H$. To prove the converse inclusion
take any $(a, b) \in H\oplus G/H$. Then there exists an element ${\tilde b} \in G$ such that $\pi|_G({\tilde b}) = b$. Now
by the surjectivity of $F$ (see the proof of \ref{ZmedSplitsCoarsely} and
\ref{ZmedPowerNSplitsCoarsely}) there exists a ${\tilde c} \in {\mathbb G}$ such that $F({\tilde c}) = (a, b)$ but it would imply ${\tilde c} - {\tilde b} \in H$ thus $c \in G$
and we have proved $H \oplus G/H \subseteq F|_G(G)$.
\end{proof}
\begin{Problem}Classify even-by-quotient groups of infinite asymptotic dimension up to coarse equivalence.
\end{Problem}
\section{Coarse classification and rigidity}\label{MainSection}
Next theorem is derived form the results of the previous two sections.
\begin{Thm} \label{SuperTheorem}
Let $G$ be an abelian group such that $\dim_{\QQ}(G\otimes \QQ) < \infty$ and let $H$ be a subgroup of $G$ generated by a maximal independent free system of elements of $G$
then the exact sequence
\[0\to
H \overset{i}{\rightarrow} G \overset{\pi}{\to} G/H\to 0\]
is coarsely split. Hence $G$ is coarsely equivalent to $H\oplus G/H$.
\end{Thm}
\begin{proof}
Let $T_2$ be the $2$-Torsion subgroup of $G/H$ i.e. the subgroup of $G/H$ such that all the elements has order some power of $2$. Notice that $G/H$ is locally finite. 
We get that the group ${\mathbf T_2}$ defined as ${\mathbf T_2} := \pi^{-1}(T_2)$ is an even-by-quotient group and by theorem
\ref{EvenByQuotientSplitsCoarsely} we have that the exact sequence
\[0\to
H \overset{i}{\rightarrow} \T2  \overset{\pi}{\to} T_2 \to 0\] is coarsely split.
Hence there exists some proper left invariant metrics $d_H$, $d_{\T2}$ and $d_{T_2}$ and a coarse equivalence
$f: (\T2, d_{\T2}) \to (H\oplus T_2, d_H\oplus d_{T_2})$ (notation: $f(x) = (f_1(x), f_2(x))$). Now we have that $(G,\T2)$ and $(G/H, T_2)$
are odd pairs
and by Theorem \ref{OddPairsSplitsCoarsely} the exact sequences:
\[0\to
\T2    \overset{i}{\rightarrow} G  \overset{\pi}{\to} G/\T2   \to 0\]
\[0\to
T_2 \overset{i}{\rightarrow} G/H  \overset{\pi}{\to} G/\T2   \to 0\]
are coarsely split. So there exist two proper left invariant metrics
$d_G$ and $d_{G/\T2}$ and two bi-uniform equivalences
\[{\tilde f}: (G, d_G)\to (\T2\oplus G/\T2, d_{\T2}\oplus d_{G/\T2}) \text{  (notation:  } {\tilde f(x)} = ({\tilde{f}}_1(x), {\tilde f}_2(x))\text{ )}\]
\[{\tilde g} : (G/H, d_{G/H}) \to (T_2 \oplus G/\T2, d_{T_2}\oplus d_{G/\T2})\text{  (notation: }{\tilde g}(x) = ({\tilde g}_1(x), {\tilde g}_2(x))\text{ )}\]
Define the map $F: (G, d_G) \to (H\oplus G/H, d_H \oplus d_{G/H})$  as $F(x) = (F_1(x), F_2(x))$ with $F_1$ and $F_2$ given by:
\[F_1(x) = f_1(\tf_1(x))\]
\[F_2(x) = \tg^{-1}(f_2(\tf_1(x)), \tf_2(x))\]
It is clearly a coarse equivalence and also if $x \in H$ then $F(x) = (x, 0)$. By the construction of $f$, $\tf$ and $\tg$ we obtain that
$F_2(x) = x + H = \pi(x)$ for every $x \in G$. So the original exact sequence is coarsely split.
\end{proof}

As a corollary we get easily the following result of Dranishnikov and Smith about the asymptotic
dimension of $G$.

\begin{Cor}[Dranishnikov-Smith, \cite{Dran-Smith}] Let $G$ be a countable abelian group then $asdim (G) = dim_{\QQ} (G \otimes \QQ)$.
\end{Cor}
\begin{proof}
If $\dim_{\QQ}(G\otimes \QQ)$ is finite the result is a direct consequence of the previous theorem, the fact $\asdim (H\oplus G/H) \le \asdim H + \asdim G/H$ and the result 
that says that all locally finite group are of asymptotic dimension zero(see \cite{Smith} and \cite{Brod-Dydak-Higes-Mitra}).
If it is infinite just notice that $\bigoplus_{i=1}^{\infty} \ZZ \subset G$ and $\asdim (\bigoplus_{i=1}^{\infty} \ZZ)= \infty$.
\end{proof}
From theorem \ref{SuperTheorem} we get that all countable abelian groups of finite torsion free rank are coarsely equivalent to a group $\ZZ^n\oplus L$ with 
$L$ a torsion group. So in view of this, corollary \ref{FinitelyGeneratedAreConnected} and corollary \ref{TorsionAbelian} we get that the torsion free rank of a countable group 
and the finitely-non finitely generated property classify up to coarse equivalence abelian groups. So we are interested in finding an algebraic coarse invariant that 
together with the torsion free rank determines if a group is finitely generated or not. A good candidate is the $\QQ$-cohomological dimension. 
To show that we will need the following theorem that is a trivial consequence of the results of \cite{Gil-Str}. 
\begin{Thm}(Gildenhuys-Strebel \cite{Gil-Str}) If $G$ is a countable torsion free abelian group then:
\begin{enumerate}
\item $\dim_{\QQ}(G\otimes \QQ)\le cd_{\QQ}(G) \le \dim_{\QQ}(G\otimes \QQ) +1$.
\item If $cd_{\QQ} (G) = \dim_{\QQ}(G\otimes \QQ)$ then $G$ is finitely generated.
\end{enumerate}
\end{Thm}

Combining this result, the Shalom-Sauer theorem \ref{Shalom-Sauer} and our results we get:
\begin{Cor}\label{CohomologicalDimensionAbelian} Let $G$ be a countable abelian group such that $\dim_{\QQ}(G\otimes \QQ)< \infty$ . 
Then:
\begin{enumerate}
\item $G$ is finitely generated if and only if $cd_{\QQ}(G) = \dim_{\QQ}(G\otimes \QQ)$.     
\item If $G$ is not finitely generated then $cd_{\QQ}(G) = \dim_{\QQ}(G\otimes\QQ) +1$.
\end{enumerate}
\end{Cor}
\begin{Rem}As said in the introduction from the Shalom-Sauer theorem and our results to estimate the $R$-(co)homological dimension of countable abelian groups 
it will be enough for certain rings to calculate the dimension of $\ZZ^n \oplus \bigoplus_{i=1}^{\infty} \ZZ_2$.
\end{Rem}
Finally we obtain the two main theorems of this paper:

\begin{Thm}\label{MainTheorem}
Let $G$ and $H$ be two countable abelian groups. Then we have: 
\begin{enumerate}
\item If $\dim_{\QQ}(G\otimes \QQ) < \infty$ then $G$ is coarsely equivalent to $H$ if and only if 
$\dim_{\QQ}(G\otimes \QQ) = \dim_{\QQ}(H\otimes \QQ)$ and $cd_{\QQ}(G) = cd_{\QQ}(H)$.
\item If $\dim_{\QQ}(G\otimes \QQ) = \infty$ then $G$ is coarsely equivalent to $A \oplus K$ with $K$ some(any) locally finite group and $A$ a subgroup of
$\bigoplus_{i=1}^{\infty} \Zmed \oplus \bigoplus_{i=1}^{r_2(G)}\Ztwo$.
\item If $H$ is finitely generated but $G$ is not finitely generated then there is a coarse embedding from $G$ to $H$ if and only if
 $\dim_{\QQ}(G\otimes \QQ) < \dim_{\QQ}(H\otimes \QQ)$.
\end{enumerate}
\end{Thm}
\begin{proof}
(1) is a consequence of the previous discussion. \par To prove (2) we do an analogous reasoning as in the proof of \ref{SuperTheorem}. Assume $G$ is a group 
of infinite torsion free rank and let $A$ be the maximal subgroup of $G$ generated by a maximal independent free system of elements of $G$. Let $T_2$ be 
the $2$-torsion group of $G/A$. We have that $\pi^{-1}(T_2) = {\mathbf T_2}$ is even-by-quotient and that $(G, {\mathbf T_2})$ is an odd-pair. 
Applying theorem \ref{OddPairsSplitsCoarsely} and proposition \ref{EvenCanBeEmbedded} we get easily (2).\par
In \cite{Brod-Dydak-Higes-Mitra} it was proved that any locally finite group can be embedded coarsely in $\ZZ$. Now if 
$\dim_{\QQ}(G\otimes \QQ) < \dim_{\QQ}(H\otimes \QQ)$ then $G$ is coarsely equivalent to $\ZZ^n\oplus L$, $H$ is coarsely equivalent to $\ZZ^m$ and $m>n$. So 
it is clear 
that there is a corsely embedding of $\ZZ^n\oplus L$ in $\ZZ^m$. The converse is a direct consequence of the Shalom-Suaer theorem and corollary \ref{CohomologicalDimensionAbelian}
   
\end{proof} 

\begin{Thm}\label{Rigidity} If $G$ is a countable group coarsely equivalent to an abelian group then $G$ 
is locally virtually abelian.
\end{Thm}
\begin{proof}
By theorem \ref{SuperTheorem} and theorem \ref{MainTheorem} we can assume that $G$ is coarsely equivalent to $\ZZ^n \oplus \bigoplus_{i=1}^K \ZZ_{p_i}$. If $K < \infty$ 
the result is well known as mentioned in the introduction. Let us assume that $K = \infty$. Let $f: G\to \ZZ^n \oplus \bigoplus_{i=1}^{\infty} \ZZ_{p_i}$ 
be a coarse equivalence with 
$g: \ZZ^n \oplus \bigoplus_{i=1}^{\infty} \ZZ_{p_i} \to G$ its coarse inverse. Let $H$ be a finitely generated subgroup of $G$. As a consequence of 
Corollary \ref{FinitelyGeneratedAreConnected} 
we get that there exists an $m \in \NN$ such that $f(H) \subset \ZZ^n \oplus \bigoplus_{i=1}^m \ZZ_{p_i}$. 
Define $H'$ as the minimum subgroup finitely generated that contains 
$H$ and $g(\ZZ^n \oplus \bigoplus_{i=1}^m \ZZ_{p_i})$, by lemma \ref{ComponentsInCountableGroups} such group exists. 
Applying again corollary \ref{FinitelyGeneratedAreConnected}
we obtain that there exists an $m'\ge m$ such that $f(H') \subset \ZZ^n \oplus \bigoplus_{i=1}^{m'} \ZZ_{p_i}$. 
Now as $f$ is a coarse equivalence $f|_{H'}$ is a coarse embedding. 
It is obvious that there exists a constant $C >0$ such that $\ZZ^n \oplus \bigoplus_{i=1}^{m'} \ZZ_{p_i} \subset B_C(\ZZ^n \oplus \bigoplus_{i=1}^m \ZZ_{p_i})$. 
As $g$ is the inverse of $f$ it is also clear that there exists $C'>0$ such that 
$\ZZ^n \oplus \bigoplus_{i=1}^m \ZZ_{p_i} \subset B_{C'}(f(g(\ZZ^n \oplus \bigoplus_{i=1}^m \ZZ_{p_i})))$. Therefore the map 
$f|_{H'}: H' \to \ZZ^n \oplus \bigoplus_{i=1}^{m'} \ZZ_{p_i}$ is a coarse equivalence and $H'$ is virtually abelian. 
It is not hard to see that any subgroup of a virtually abelian 
group is virtually abelian. It implies that $H$ is virtally abelian.  
\end{proof}
\begin{Problem}\label{RigidityProblem} Is any locally virtually abelian group of finite asymptotic dimension coarsely equivalent to an abelian group?
\end{Problem}

\end{document}